\documentclass{amsart}
\usepackage{amssymb}
\usepackage{pb-diagram}
\usepackage{color}
\usepackage[normalem]{ulem}

\usepackage[textsize=tiny]{todonotes}

\newtheorem{theorem}{Theorem}[section]

\newtheorem{fact}[theorem]{Fact}
\newtheorem{cor}[theorem]{Corollary}
\newtheorem{proposition}[theorem]{Proposition}
\newtheorem{lemma}[theorem]{Lemma}
\newtheorem{corollary}[theorem]{Corollary}

\newtheorem{question}[theorem]{Question}
\newtheorem{conjecture}[theorem]{Conjecture}

\theoremstyle{definition}
\newtheorem{defn}[theorem]{Definition}
\newtheorem{definition}[theorem]{Definition}
\newtheorem{example}[theorem]{Example}

\theoremstyle{remark}

\newtheorem{remark}[theorem]{Remark}

\newcommand{\la}{\langle}
\newcommand{\ra}{\rangle}
\newcommand{\CM}{{\cal M}}

\newcommand{\sub}{\subseteq}

\newcommand{\CK}{\cal K}

\newcommand{\CR}{{\cal R}}

\newcommand{\dcl}{\operatorname{dcl}}
\newcommand{\acl}{\operatorname{acl}}
\newcommand{\scl}{\operatorname{scl}}
\newcommand{\cl}{\operatorname{cl}}
\newcommand{\cld}{\operatorname{cl}^{\delta}}

\newcommand{\ldim}{\ensuremath{\textup{ldim}}}

\newcommand{\ddim}{\text{$\delta$-\textup{dim}}}

\newcommand{\bb}[1]{\ensuremath{\mathbb{#1}}}

\newcommand{\cal}[1]{\ensuremath{\mathcal{#1}}}

\newcommand{\Lrarr}{\ensuremath{\Leftrightarrow}}
\newcommand{\Rarr}{\ensuremath{\Rightarrow}}

\newcommand{\sm}{\setminus}

\newcommand{\R}{\mathbb{R}}

\title[Coincidence of dimensions in closed ordered differential fields]{Coincidence of dimensions in closed ordered differential fields}
\date{\today}

\begin{document}
\author {Pantelis  E. Eleftheriou}
\address{Pantelis  E. Eleftheriou, Department of Mathematics and Statistics, University of Konstanz, Box 216, 78457 Konstanz, Germany}
\email{panteleimon.eleftheriou@uni-konstanz.de}

\author{Omar Le\'on S\'anchez}
\address{Omar Le\'on S\'anchez\\
University of Manchester\\
Department of Mathematics\\
Oxford Road \\
Manchester, M13 9PL.}
\email{omar.sanchez@manchester.ac.uk}

\author {Nathalie Regnault}
\address{Nathalie Regnault, Department of Mathematics (De Vinci), UMons, 20, place du Parc 7000 Mons, Belgium}
\email{nathalie.regnault@umons.ac.be}

\thanks{The first author  was supported by a Research Grant from the German Research Foundation (DFG) and a Zukunftskolleg Research Fellowship.}

\begin{abstract} Let $\CK=\la\cal R, \delta\ra$ be a closed ordered differential field, in the sense of Singer \cite{singer}, and $C$ its  field of constants. In this note, we prove that, for sets definable in the pair $\cal M=\la \cal R, C\ra$, the $\delta$-dimension from \cite{bmr} and the large dimension from \cite{egh} coincide. As an application, we characterize the definable sets in $\CK$ that are internal to $C$ as those sets that are definable in $\cal M$ and have $\delta$-dimension $0$. We further show that, for sets definable in $\CK$, having $\delta$-dimension $0$ does not generally imply co-analyzability in $C$ (in contrast to the case of transseries). We also point out that the coincidence of dimensions also holds in the context of differentially closed fields and in the context of transseries.
\end{abstract}

\maketitle

\section{Introduction}

Pairs of fields have been extensively studied in model theory and arise naturally in various ways. If $\CK=\la \cal R, \delta\ra$ is a differentially closed field of characteristic zero (DCF$_0$), a closed ordered differential field (CODF) or the differential field of transseries $\mathbb T$ constructed in \cite{Asch}, and $C=\ker(\delta)$ is the field of constants in each case, then the reduct $\cal M=\la \cal R, C\ra$ is a pair of algebraically closed fields (\cite{keisler}, \cite{poizat}), a dense pair of real closed fields (\cite{vdd-dense}, \cite{rob}) or a tame pair (\cite{mac}, \cite{dl}), respectively. In all three cases, there is a natural notion of dimension, the \emph{differential} or \emph{$\delta$-dimension} for definable sets in $\CK$. While we postpone the definitions until \S\ref{preliminaries}, we do recall that in \cite[Corollay 5.27]{bmr} it is shown that, in the case of CODFs, the $\delta$-dimension on definable sets coincides with the one obtained from $\delta$-cell decomposition.


\medskip

In the above three cases (DCF$_{0}$, CODF, and transseries), the following implications hold for a set $X$ definable in $\CK$:
\begin{equation}
 \text{$X$ is internal to $C$}\,\,\, \Rarr\,\,\, \text{$X$ is co-analyzable in $C$} \,\,\, \Rarr\,\,\, \text{$X$ has $\delta$-dimension $0$}.\notag
\end{equation}
In the case of transseries, it is shown in \cite{adh} that the latter two properties are equivalent, and in \cite{dl} that, 
when restricted to definable sets in the pair $\cal M=\la \cal R, C\ra$, all three properties are equivalent. In this  note, we prove that  in the case of CODFs the latter two properties are different (see examples in \S \ref{example}), whereas when restricted to sets definable in $\cal M$ all three properties are  equivalent. The latter result is a consequence of Theorem \ref{main} below, which states that  in $\cal M$, the $\delta$-dimension coincides with the \emph{large dimension} from \cite{egh}, which we briefly describe next (and provide further details in \S\ref{preliminaries}).

\medskip

The primary example of a dense pair of real closed fields is that of $\la \overline{\mathbb R}, \mathbb Q^{rc}\ra$, where $\overline{\mathbb R}$ is the real field and $\mathbb Q^{rc}$ the subfield of real algebraic numbers. This pair was studied by A. Robinson in his classical paper \cite{rob}, where the decidability of its theory was proven. A systematic study of dense pairs $\la \cal B, \cal A\ra$ of  o-minimal structures (that is, $\cal A$ is a dense proper elementary substructure of \cal B) occurred much later in \cite{vdd-dense} by van den Dries. In \cite{egh} (as well as \cite{beg} and \cite{for}), the pregeometric dimension localized at $\cal A$ was studied. Namely, letting $A$ and $B$ be the underlying sets of $\cal A$ and $\cal B$, respectively, and $\dcl$ denote the usual definable closure in the o-minimal structure $\cal B$, we set, for $D\sub B$,
$$\scl(D)=\dcl(D\cup A).$$
For $X$ a definable set in $\la \cal B, \cal A\ra$, the $\scl$-dimension of $X$ is called \emph{large dimension} and is denoted by $\ldim(X)$. In \cite{egh},  in a much broader setting that includes dense pairs, the large dimension was given a topological description via a structure theorem,
much alike the topological description of the usual $\dcl$-dimension via the cell decomposition theorem in the o-minimal setting. 

\medskip


The main result of this note is the following theorem.


\begin{theorem}\label{main} Let $\cal K=\la \cal R, \delta\ra$ be a closed ordered differential field and $C$ its field of constants. Let $X$ be a set definable in the pair $\cal M=\la \cal R, C\ra$. Then
$$\text{$\ldim (X)= \delta$-$\dim (X)$.}$$
\end{theorem}

With a bit more work, we can characterize the notion of being internal to $C$ for sets definable in $\CK$.

\begin{cor}\label{cor} Let $\cal K=\la \cal R, \delta\ra$ be a closed ordered differential field and $C$ its field of constants. Let $X$ be a set definable in $\CK$. Then
$$\text{$X$ is internal to $C$} \,\,\Lrarr\,\,\,  \text{$\delta$-$\dim(X)=0$ and $X$ is definable in $\cal M=\la \cal R, C\ra$.}$$
\end{cor}


In Section~\ref{example}, in contrast to the case of transseries \cite[Proposition 6.2]{adh}, we give several examples of sets definable in $\CK$ with $\delta$-dimension $0$ which are \emph{not} co-analyzable in $C$ (and hence also \emph{not} internal to $C$). Finally, in Section~\ref{dimdcf}, we prove the analogues of Theorem~\ref{main} in the realm of differentially closed fields of characteristic zero and in the realm of transseries.

\medskip


\noindent {\bf Acknowledgements.} We are grateful to Anand Pillay for several comments and suggestions that led to a significant improvement of a previous draft.

\section{Preliminaries}\label{preliminaries} Throughout we fix a real closed field $\cal R=\la R,<, +, \cdot\ra$ and a derivation $\delta$ on $R$. Also, for us, definability is always meant with parameters.

\subsection{Closed ordered differential fields}\label{codfs} The differential field $\cal K=\la \cal R, \delta\ra$ is a CODF if it is existentially closed among ordered differential fields. The class of CODFs is first-order axiomatisable in the language of ordered differential rings; for instance, Singer's axioms \cite{singer} state that $\CK\models$ CODF if and only if
\begin{enumerate}
\item [($\dagger$)] for differential polynomials $f,g_1,\dots,g_m\in K\{x\}$ with $ord(f)=n\geq ord(g_i)$, if there is $(a_0,a_1,\dots,a_n)\in R^{n+1}$ such that as algebraic polynomials
$$f(a_0,\dots,a_n)=0, \; \frac{\partial f}{\partial \delta^n x}(a_0,\dots,a_n)\neq 0, \; \text{ and } \; g_i(a_0,\dots,a_n)>0,$$
then there is $b\in K$ such that as differential polynomials
$$f(b)=0 \; \text{ and }\;  g(b)\neq 0.$$
\end{enumerate}

The theory CODF admits quantifier elimination \cite{singer} and elimination of imaginaries \cite{point} (in the language of ordered differential rings). Furthermore, by \cite[\S1.4.2]{brouette}, for any model $\CK=\la \cal R,\delta\ra$ and $A\sub R$ the definable and algebraic closure $\dcl^{\CK}(A)=\acl^{\CK}(A)$ equals the real closure of the differential field generated by $A$.



\medskip

\noindent \textbf{Notation.} Until Section \ref{dimdcf}, we fix a CODF $\CK=\la\cal R,\delta\ra$, with $C$ its field of contants, and $\cal M=\la \cal R, C\ra$. By \cite{bmr}, the reduct $\cal M$ is a dense pair of real closed fields.

\medskip

We now prove a preliminary result (Lemma \ref{Cinternal}) that will be used in \S\ref{mainproofs}.  Let $\cal N$ be any of $\CK, \CM, \CR$, and let $X$ be an arbitrary (not necessarily definable) subset of $R^m$. We say that $X$ is \emph{internal to $C$} in $\cal N$ if there is $n>0$ and a function $f:R^n \to R^m$ definable in $\cal N$ such that $X\sub f(C^n)$. In Lemma \ref{Cinternal}, we prove that for definable sets in $\CK$, internality to $C$ in $\cal N$ is invariant under varying $\cal N$ among $\CK$, $\CM$ and $\CR$. We will need  the following result on definable functions in a CODF\footnote{We thank Marcus Tressl for pointing out the argument in the proof.}.

\begin{lemma}\label{definablefunctions}
If $f:R^n\to R$ is a definable function in $\CK$, then there is $d\geq 0$ and a function $F:R^{n(d+1)}\to R$ definable in $\cal R$ such that
$$f(a)=F(a,\delta a,\dots,\delta^d a) \quad \textrm{ for all }a\in R^n.$$
\end{lemma}
\begin{proof}
Since for any $A\sub R$ the definable closure $\dcl^{\CK}(A)$ equals the real closure of the differential field generated by $A$, it follows, by a standard compactness argument, that there is a partition $X_1,\dots, X_r$ of $R^n$ into sets definable in $\CK$ and, for some $d\geq 0$, there are functions $F_i:R^{n(d+1)}\to R$ definable in $\cal R$, for $i=1,\dots,r$, such that
$$f(a)=F_i(a,\delta a,\dots,\delta^d a) \quad \textrm{ for all }a\in X_i.$$
By quantifier elimination and after possibly increasing $d$, for each $i$, there is a set $T_i\sub M^{n(d+1)}$ definable in $\cal R$ such that
$$X_i=\{a\in R^n: (a,\delta a,\dots,\delta^d a)\in T_i\}.$$
Now define $F:R^{n(d+1)}\to R$ as
$$
F(b)=
\left\{
\begin{array}{cc}
F_1(b) & \text{ if } b\in T_1  \\
F_2(b) & \text{ if } b\in T_2\setminus T_1 \\
\vdots & \vdots \\
F_r(b) & \text{ if } b\in T_r\setminus (T_1\cup \cdots \cup T_{r-1}) \\
0   &    \text{ otherwise}
\end{array}
\right.
$$
This $F$ is the desired function.
\end{proof}

\begin{lemma}\label{Cinternal}
Assume $X\sub R^m$ is definable in $\CK$. If $X$ is internal to $C$ in $\CK$, then $X$ is definable in the pair $\cal M=\la \cal R, C\ra$ and internal to $C$ in $\cal R$.
\end{lemma}
\begin{proof}
Since $X$ is internal to $C$ in $\CK$, there is $n>0$ and a function $f:R^n\to R^m$ definable in $\CK$ such that $X\subseteq f(C^n)$. Since $X$ is definable in $\CK$, after possibly modifying $f$, we may assume $X=f(C^n)$. By Lemma~\ref{definablefunctions}, there is $d\geq 0$ and a function $F:R^{n(d+1)}\to R$ definable in $\cal R$ such that
$$f(a)=F(a,\delta a,\dots,\delta^d a) \quad \textrm{ for all }a\in R^n.$$
Hence, $f(a)=F(a,0,\dots,0)$ for all $a\in C^n$. This shows that $X$ is defined by
$$X=\{y\in R^m:\; \exists x\in C^n \textrm{ with } y=F(x,0,\dots,0) \}$$
Thus, $X$ is definable in $\cal M$ and internal to $C$ in $\cal R$.
\end{proof}
In view of Lemma \ref{Cinternal}, we say that a set definable in $\CK$ is \emph{$C$-internal} if it is so in any of $\CK, \CM, \CR$.


\smallskip

We now discuss the notion of differential dimension (or $\delta$-dimension). For $A\subseteq R$ we let $\cld(A)$ denote the set of elements of $R$ that are differentially algebraic over $A$; namely, $a\in \cld(A)$ if and only if $a$ is a solution of a nonzero differential polynomial in one variable with coefficients in the differential field generated by $A$. One can also work over a set of parameters $B\subseteq R$ and set $\cld_B(A)=\cld(A\cup B)$. This closure operator defines a pregeometry (localised at $B$) on $R$, called the $\cld_B$-pregeometry of $\CK$. Hence, we can induce a $\cld_B$-dimension on finite tuples; namely, given $a=(a_1,\dots,a_n)\in R^n$ we set $\cld_B$-$\dim(a)$ to be the maximal cardinality of a $\cld_B$-independent subtuple of $a$. That is, $\cld_B$-$\dim(a)=k$ if and only if, after possibly re-ordering the tuple, $a_i\notin \cld_B(a_1,\dots,a_{i-1})$ for $1\leq i\leq k$ and $a_j\in \cld_B(a_1,\dots,a_k)$ for $k<j\leq n$. It follows that $\cld_B$-$\dim(a)$ equals the differential transcendence degree over $B$ of the differential field generated by $a$.

\smallskip

For definable sets in $\CK$ we define $\delta$-dimension as follows.

\begin{definition}\label{defdeltadim}
Let $\CK^*=\la \cal R^*,\delta \ra$ be a $|R|^+$-saturated elementary extension of $\CK$. For any nonempty set $X$ definable in $\CK$, we set $X^*=X(R^*)$ (that is, the realisations in $\CK^*$ of any formula over $R$ defining $X$ in $\CK$). The \emph{$\delta$-dimension} is defined as
$$\delta\text{-}\dim(X)=\max_{a\in X^*}\; \cld_R\text{-}\dim(a),$$
where $\cld_R$-$\dim(a)$ is taken with respect to the $\cld_R$-pregeometry of $\CK^*$. We note that the definition is independent of the choice of the (suitably saturated) elementary extension $\CK^*$ of $\CK$. For matter of convenience one sets $\delta$-$\dim(\emptyset)=-\infty$.
\end{definition}

In \cite[Corollay 5.27]{bmr} it was shown that, for definable sets in $\CK$, this notion of $\delta$-dimension coincides with the one obtained from $\delta$-cell decomposition. Furthermore, in \S5.3 of the same paper, it was shown that $\delta$-dimension is a \emph{dimension function} in the sense of \cite{vdd-dim}. The following summarizes properties shared by all dimension functions.

\begin{fact}\label{dimprop} \cite[\S1]{vdd-dim} Let $X\subseteq R^n$ and $Y\subseteq R^m$ be definable in $\CK$.
\begin{enumerate}
\item If $X$ is a finite nonempty set, then $\delta$-$\dim(X)=0$.
\item $\delta$-$\dim(R^n)=n$.
\item If $n=m$, then $\delta$-$\dim(X\cup Y)=\max\{\delta$-$\dim(X), \; \delta$-$\dim(Y)\}$.
\item $\delta$-$\dim(X\times Y)=\delta$-$\dim(X)+\delta$-$\dim(Y)$.
\item if $\delta$-$\dim(X)=k$, then there is a coordinate projection $\pi$ such that $\delta$-$\dim\pi(X)=k$.
\item Let $f:R^n\to R^m$ be a function definable in $\CK$. Then $\delta$-$\dim f(X)\leq \delta$-$\dim(X)$. Furthermore, for $0\leq i\leq n$, if we let
$$X_i=\{r\in R^m: \delta\text{-}\dim f^{-1}(r)=i\},$$
then $X_i$ is definable in $\CK$ and
$$\delta\text{-}\dim X_i \;+\; i=\delta\text{-}\dim f^{-1}(X_i).$$
\item $\delta$-dim is completely determined by its effect on subsets of $R$ definable in $\CK$.
 \end{enumerate}
\end{fact}

\begin{remark}
We note that the definition of dimension function only requires properties (1), (2), (3), (6); and so properties (4), (5), (7) are consequences of these. This is all justified in \S1 of \cite{vdd-dim}. 
\end{remark}

\medskip

\subsection{Notions from dense pairs}\label{densepair}
In this subsection we recall the necessary material on large dimension from \cite{egh}. Recall that $\CK=\la \cal R,\delta\ra$ is a CODF, with $C$ its field of constants, and  $\cal M=\la \cal R,C\ra$. Let $\dcl$ denote definable closure in the o-minimal structure $\cal R$. For $A\subseteq R$, we set
$$\scl(A):=\dcl(A\cup C).$$
We can work over a parameter set $B\subseteq R$ and let
$$\scl_B(A):=\scl(A\cup B)=\dcl(A\cup B\cup C).$$
This closure operator defines a pregeometry (localised at $B$) on $R$, called the $\scl_B$-pregeometry of $\cal M$. In the usual manner, this induces a $\scl_B$-dimension on finite tuples  of $R$. For sets definable in $\cal M$ we can define a dimension, as follows:

\begin{definition}
Let $\cal M^*=\la \cal R^*,C^* \ra$ be a $|R|^+$-saturated elementary extension of $\cal M$. For any nonempty set $X$ definable in $\cal M$, we set $X^*=X(R^*)$ (i.e., the realisations in $\cal M^*$ of any formula over $R$ defining $X$ in $\cal M$). The \emph{large dimension}, denoted by $\ldim$, is defined as
$$\ldim(X)=\max_{a\in X^*}\; \scl_R\text{-}\dim(a),$$
where $\scl_R$-$\dim(a)$ is taken with respect to the $\scl_R$-pregeometry of $\cal M^*$. This is independent of the choice of the (suitably saturated) elementary extension $\cal M^*$ (see for instance \cite[\S2]{vdda}). Finally, one sets $\ldim(\emptyset)=-\infty$.
\end{definition}

As mentioned in the introduction, in \cite{egh} the large dimension was characterised topologically via a structure theorem
. Furthermore, in Lemma 6.11 of the same paper, it was shown that large dimension is a dimension function. Therefore large dimension satisfies properties (1)-(7) of Fact~\ref{dimprop} (stated there for $\delta$-$\dim$). This is one of the key ingredients in the proof of Theorem~\ref{main}. The other key ingredient is a characterisation, obtained in \cite{egh}, of definable sets in $\cal M$ of large dimension zero. This is given in terms of the following notion (essentially coming from \cite{dg}).

\begin{defn}\label{def-small}
Let $X\sub R^n$ be a set definable in $\cal M$. We call $X$ \emph{large} if there is $m\geq 0$ and a function $f:R^{nm}\to R$ definable in $\cal R$ such that $f(X^m)$ contains an open interval in $R$. We call $X$ \emph{small} if it is not large.
\end{defn}

\begin{fact}\label{fact-small}\cite[Corollary 3.11 and Lemma 6.11(3)]{egh}
Let $X$ be a set definable in $\cal M$. Then,
$$\ldim(X)=0\;  \iff \; X \text{ is small } \;  \iff \;  X \text{ is internal to C }$$
\end{fact}


\section{Proof of Theorem \ref{main}}\label{mainproofs}

We prove Theorem~\ref{main} through a series of lemmas.

\begin{lemma}\label{ddimC} 
$\ddim \; C=0$.
\end{lemma}
\begin{proof}
Let $\CK^*=\la \cal R^*,\delta\ra$ be an $|R|^+$-saturated elementary extension of $\CK=\la \cal R, \delta\ra$ and $C^*=C(R^*)$, see Definition~\ref{defdeltadim}. Any point $a$ in $C^*$ is a solution of the nonzero differential polynomial $\delta x$. Thus the differential transcendence degree of $a$ is zero. 
\end{proof}

\begin{lemma}\label{0d} Let $X\sub R^n$ be a set definable in $\cal M$. If $\ldim X=0$, then $\ddim X=0$.
\end{lemma}
\begin{proof}
Suppose $\ldim X=0$. By  Fact \ref{fact-small}, $X$ is internal to $C$; namely, there is a function $f:R^m\to R^n$ definable in $\cal M$, such that $X\subseteq f(C^m)$.  By Lemma \ref{ddimC} and (3), (4), (6) of Fact~\ref{dimprop},
$$\ddim X\leq  \ddim f(C^m)\le \ddim C^m =0.$$
\end{proof}


The following lemma is a useful fact on CODFs.

\begin{lemma}\label{0dC}
Let $p(x)$ be a nonzero differential polynomial over $R$. The solution set $\{c\in R: p(c)=0\}$ is co-dense in $R$ with respect to the order topology.
\end{lemma}
\begin{proof}
This fact follows from the axioms of CODFs stated in ($\dagger$) in \S\ref{codfs}, and  appears to be folklore. We provide the details. Let $(a,b)$ be a nonempty open interval, we prove that there is $c\in (a,b)$ such that $p(c)\neq 0$. Let $f(x)=\delta^m x$ with $m>\textrm{ord}(p(x))$, $g_0(x)=p(x)$, $g_1(x)=x-a$ and $g_2(x)=b-x$ .
One can find $\bar c=(c_0,\dots,c_m)\in R^m$ such that, as algebraic polynomials, $f(\bar c)=0$, $g_0(\bar c)\neq 0$, $g_1(\bar c)>0$ and $g_2(\bar c)>0$. By the axioms ($\dagger$), we can find $c\in R$ such that, as differential polynomials, $g_0(c)\neq 0$, $g_1(c)>0$ and $g_2(c)>0$. This yields $p(c)\neq 0$ and $c\in (a,b)$, as desired.
\end{proof}

\begin{lemma}\label{ld}
 Let $X\sub R^n$ be a set definable in $\cal M$. If $\ddim X=0$, then $\ldim X=0$.
\end{lemma}
\begin{proof}
Suppose towards a contradiction that $\ldim X>0$. By Fact~\ref{fact-small}, $X$ is large, and hence, by definition,
there is a function $f:R^{nm}\to R$ definable in $\cal R$ such that $f(X^m)$ contains an open interval in $R$. On the other hand, by (4) and (6) of Fact~\ref{dimprop}, we have
$$\ddim f(X^m)\le \ddim X^m=0.$$
By compactness, there must exist a nonzero differential polynomial $p(x)$ over $R$ that vanishes in all of $f(X^m)$. As the latter contains an interval, this contradicts Lemma \ref{0dC}.
\end{proof}

We can now prove Theorem~\ref{main} and Corollary~\ref{cor}.

\begin{proof}[Proof of Theorem \ref{main}]
By property (7) in Fact~\ref{dimprop} of dimension functions, $\ddim$ and $\ldim$ will coincide in all sets definable in $\cal M$ as long as they coincide on definable subsets of $R$. By (2) in Fact~\ref{dimprop}, $\ddim(R)=1=\ldim(R)$, and so any nonempty definable subset of $R$ has $\ddim$ and $\ldim$ either 0 or 1. Hence, it suffices to show that for definable $X\subseteq R$ we have
$$\ddim(X)=0 \; \iff \; \ldim(X)=0.$$
But this follows by putting together Lemmas \ref{0d} and \ref{ld}.
\end{proof}

\begin{proof}[Proof of Corollary \ref{cor}]
Left-to-right is by Lemma \ref{Cinternal}. For right-to-left, by Theorem \ref{main}, $\ldim(X)=0$, and, by Fact \ref{fact-small}, $X$ is internal to $C$.
\end{proof}

\section{Examples of zero dimensional not co-analyzable sets}\label{example}

In this section we exhibit examples of sets definable in $\CK$ with $\delta$-dimension $0$ which are not co-analyzable in $C$ (and hence also not internal to $C$). We recall that in the case of transseries no such example exists \cite[Proposition 6.2]{adh}. Our examples come from classical constructions of strongly minimal sets in DCF$_0$ that are orthogonal to the constants (over suitable parameter sets). We carry on our terminology from previous sections. Furthermore, for convenience, here we work in a universal (sufficiently saturated) model $\CK=\la\cal R,\delta\ra$ of CODF with constants $C$.

\smallskip


We begin by recalling the definition of co-analyzability in $C$ for subsets of $R^n$ (see \cite{udi} for further details). A set $X\sub R^n$ definable in $\CK$ is said to be \emph{co-analyzable in $C$ in $0$-steps} if $X$ is finite, and \emph{in $(r+1)$-steps} if there is a set $Y\subseteq C\times R^n$ definable in $\CK$ such that
\begin{enumerate}
\item the canonical projection $\pi:C\times R^n\to R^n$ maps $Y$ onto $X$, and
\item for each $c\in C$ the fibre $Y_c=\{a\in R^n: (c,a)\in Y\}$ is co-analyzable in $C$ in $r$-steps.
\end{enumerate}
\emph{Co-analyzable} means  co-analyzable in $r$-steps for some $r\geq 0$.

\bigskip



The following proposition is a type of ``orthogonality transfer principle" from DCF$_0$ to CODF. Let $F=R(i)$ and $\cal U=\la F, \delta\ra$; namely the underlying set of $\cal U$ is simply the algebraic closure of $R$ and the derivation $\delta$ is the unique extension from $R$ to $F$. From \cite{singer2}, we know $\cal U$ is a DCF$_0$. We recall that a strongly minimal set $Z$ defined in $\cal U$ is said to be orthogonal to the constants if for any parameter set $A\subset F$ over which $Z$ is defined and generic point $a\in Z$ over $A$ we have that, for any tuple $c$ of constants from $\cal U$,
$$\acl^{\cal U}(A,a) \cap \acl^{\cal U}(A,c)=\acl^{\cal U}(A).$$
Here we also recall that, for any subset $B\subset F$, the algebraic closure $\acl^{\cal U}(B)$ is just the field-theoretic algebraic closure of the differential field generated by $B$ in $F$.

\begin{proposition}\label{ZcoanC}
Let $Z\subseteq F^n$ be a set definable in $\cal U$ over parameters from $R$ that is strongly minimal. Let $X=Z(R)$ (the points of $X$ whose entries are in $R$). If $Z$ is orthogonal to the constants in $\cal U$ and $X$ is infinite, then $X$ is not co-analizable in $C$.
\end{proposition}
\begin{proof}
We first show that $X$ is not 1-step co-analysable in $C$. Towards a contradiction, suppose it is,
via a definable in $\CK$ subset $Y\subseteq C\times R^n$. Let $L$ be a small (with respect to saturation) differential subfield of $R$ over which all of $Z$, $X$, and $Y$ are defined. As $X$ is infinite, let $a\in X$ be such that $a \notin \acl^{\CK}(L)$. Recall that $\acl^{\CK}(L)$ is just the (relative) field-theoretic algebraic
closure of $L$ in $R$. Let $c \in C$ be such that $(c, a) \in Y$. In particular, $a \in \acl^{\CK}(L,c)$. As $c$ is a constant, this means that $a$ is in the field-theoretic algebraic closure of $L(c)$ in $R$ and not
in the field-theoretic algebraic closure of $L$ in $R$. This is of course also true after replacing $R$ for
$F$. As $a$ is a generic point of $Z$ over $L$ (in the sense of $\cal U$ by strong minimality of $Z$), this contradicts orthogonality of $Z$ to the constants (in $\cal U$). Thus, $X$ is not 1-step co-analyzable.

We now prove $X$ is not co-analyzable in $r+1$-steps. If it were we would have a definable in $\CK$ subset $Y\subseteq C\times R^n$ such that the projection $\pi:C\times R^n\to R^n$ maps $Y$ onto $X$ and all fibres $Y_c$ are $r$-step co-analyzable in $C$. Now take any $c\in C$ such that the fibre $Y_c$ is infinite. As $Y_c$ is contained in $X$, we can apply induction on $r$ and the same argument as above to contradict orthogonality of $Z$ to the constants.
\end{proof}

\begin{remark}
We note that the proof of Proposition~\ref{ZcoanC} relies essentially only on the fact that for $A\sub R$ the algebraic closure $\acl^\CK(A)$ equals the real closure in $R$ of the differential field generated by $A$.
\end{remark}

Proposition \ref{ZcoanC} yields plenty of examples of sets definable in $\CK$ of $\ddim$ zero and not co-analyzable in $C$. For instance, those of the form $X=Z(R)$ where $Z$ is the Manin kernel of a simple abelian variety defined over $R$ which do not descend to the constants in  $\cal U$. For more explicit examples, one can use the so-called Rosenlicht extensions \cite{rosenlicht}, as follows.

\begin{corollary}\label{forexa}
Let $f(x)$ be either $x^3-x^2$ or $\frac{x}{x-1}$. Then, the subset $X\sub R$ defined by $\delta(x)=f(x)$ is not co-analyzable in $C$.
\end{corollary}
\begin{proof}
The set $X$ is infinite; indeed, for any differential field $L$, we can define a derivation on $L(x)$ that maps $x\mapsto f(x)$. Furthermore, by a classical result of Rosenlicht (see \cite[Proposition 2]{rosenlicht}), the subset of $F$ defined by $\delta(x)=f(x)$ is orthogonal to the constants in $\cal U$. The result now follows from Proposition~\ref{ZcoanC}.
\end{proof}

We conclude this section with an application on the existence of a proper CODF extension with the same field of constants. We are not aware of any such example in the literature (and no such examples exist in the case of transseries).

  \begin{corollary}
  There is a proper extension $R\preccurlyeq S$ of CODFs with the same constants.
  \end{corollary}
\begin{proof}
This follows from the existence of a non-co-analyzable in $C$ definable set (given by Corollary \ref{forexa}, for instance) and \cite[Proposition 6.1(iv)]{adh}.
\end{proof}


\section{Analogues of Theorem~\ref{main} in DCF$_0$ and transseries}\label{dimdcf}

In this section we point out the necessary adaptations of the arguments in \S\ref{mainproofs} that are needed to prove the analogues of Theorem~\ref{main} in the case of DCF$_0$ and transseries.

\medskip

\subsection{The case of DCF$_0$.}\label{dcf}

To state the analogue of Theorem~\ref{main} in differentially closed fields of characteristic zero, we recall some properties of pairs of ACFs, for which we use \cite{vdda}. We fix an algebraically closed field $\cal F=\la F, +, \cdot\ra$ and a derivation $\delta$ on $F$ such that $\CK=\la \cal F,\delta\ra$ is a model of DCF$_0$. We  let $C$ be the field of constants of $\CK$. The structure $\cal M=\la\cal F,C\ra$ is an elementary pair of algebraically closed fields (of characteristic zero).

As in \S\ref{densepair}, the algebraic closure operator, $\acl$, of the strongly minimal structure $\cal F$ induces a pregeometry on $\cal M$; namely, for $B\subseteq F$ one sets
$$\cl_B(A)=\acl(A\cup B\cup C)$$
and this defines a pregeometry (localised at $B$) on $F$, called the $\cl_B$-pregeometry of $\cal M$. This induces a $\cl_B$-dimension on finite tuples of $F$, and for a nonempty set $X$ definable in $\cal M$, we define:
$$\dim_2(X)=\max_{a\in X^*}\; \; \cl_F\text{-}\dim (a)$$
where $X^*=X(F^*)$ and $\cal M^*=\la \cal F^*, C^*\ra$ is a $|F|^+$-saturated elementary extension of $\cal M$. One sets $\dim_2(\emptyset)=-\infty$.

In \cite{vdda} the dimension $\dim_2$ was studied (in the general context of pairs of algebraically closed fields). In the same paper they showed that this is a dimension function, and hence satisfies properties of (1)-(7) of Fact~\ref{dimprop}. On the other hand, the differential dimension, $\ddim$, can be defined in the same way as in Definition~\ref{defdeltadim} for definable sets in $\CK$. Furthermore, $\delta$-dim in $\CK$ is also a dimension function (as pointed out in \cite{Asch}).

The analogue of Theorem~\ref{main} in the DCF$_0$ setting is

\begin{theorem}\label{dcfcase}
Let $\CK=\la \cal F,\delta \ra$ be a differentially closed field of characteristic zero and $C$ its field of constants. Let $X$ be a set definable in the pair $\cal M=\la \cal F, C\ra$. Then
$$\dim_2(X)=\ddim(X).$$
\end{theorem}

The key ingredient in the proof is the following dichotomy result established in \cite[Proposition 1.1]{vdda}

\begin{fact}\label{dich}
For $X\subseteq F$ definable in $\cal M$, we have
$$\dim_2(X)\leq 0\quad \text{ or }\quad \dim_2(F\setminus X)\leq 0.$$
\end{fact}

We now prove Theorem \ref{dcfcase}. As $\dim_2$ and $\ddim$ are dimension functions, it suffices to show that they agree on subsets of $F$ definable in $\cal M$ (see (7) of Fact~\ref{dimprop}). As $\dim_2(F)=1=\ddim(F)$, it suffices to show that for definable $X\subseteq F$ we have
$$\dim_2(X)=0 \; \iff \; \ddim(X)=0.$$
By the dichotomy theorem (Fact~\ref{dich}), it actually suffices to show left-to-right in the above equivalence. So assume $\dim_2(X)=0$. Let $\CK^*=\la \cal F^*,\delta \ra$ be a $|F|^+$-saturated elementary extension of $\CK$. We must show that for any $a\in X^*=X(F^*)$ we have $\cld_F$-$\dim(a)=0$ (where the latter is taken with respect to the $\cld_F$-pregeometry of $\CK^*$). Let $C^*$ be the constants of $\CK^*$. Since $\cal M^*:=\la \cal F^*, C^*\ra$ is a $|F|^+$-saturated elementary extension of $\cal M$ and $\dim_2(X)=0$, we get $\cl_F$-$\dim(a)=0$ (where the latter is taken with respect to the $\cl_F$-pregeometry of $\cal M^*$). We thus have
$$a\in \cl_F(\emptyset)= \acl^{\cal F^*}(F, C^*)=F(C^*)^{\text{alg}}\subset F^*.$$
As all elements in $F(C^*)^{\text{alg}}$ are differentially algebraic over $F$, we get that $a\in \cld_F(\emptyset)$. Thus, $\cld_F$-$\dim(a)=0$, and so $\ddim(X)=0$.

\medskip

\subsection{The case of transseries.} Let $\CK=\la \mathbb T, \delta\ra$ be the differential field of transseries as constructed in \cite{Asch}. The field of constants in this case equals $\mathbb R$, and the reduct $\cal M=\la \mathbb T, \mathbb R\ra$ is a tame pair in the sense of \cite{dl}. The $\delta$-dimension for definable sets in $\CK$ was studied in \cite{adh} and was shown to be a dimension function. On the other hand, in \cite{vdda} the $\la \mathbb T, \mathbb R\ra$ analogue of the $\dim_2$ dimension was studied in \cite{vdda} and was shown to be a dimension function. Furthermore, putting together \cite[Proposition 4.1]{adh} and \cite[Theorem 1.4]{vdda}, we have:

\begin{fact}\label{trans}
Let $X$ be a definable set in $\cal M$. Then
$$\dim_2(X)=0 \; \iff \; X \text{is discrete } \; \iff\; \ddim(X)=0.$$
\end{fact}

The analogue of Theorem \ref{main} is

\begin{theorem}
Let $\CK=\la \mathbb T,\delta \ra$ be the differential field of transseries. Let $X$ be a set definable in the pair $\cal M=\la \mathbb T, \mathbb R\ra$. Then
$$\dim_2(X)=\ddim(X).$$
\end{theorem}
\begin{proof}
As in \S\ref{dcf}, it suffices to show that for any $X\subseteq \mathbb T$ definable in $\cal M$ we have $\dim_2(X)=0$ if and only if $\ddim(X)=0$. But this equivalence is given Fact~\ref{trans}.
\end{proof}


\end{document}